\documentclass{article}
\usepackage[comma,square,sort&compress, numbers]{natbib}
\usepackage[greek, english]{babel} 
\usepackage{amsbsy,amssymb,amscd,amsmath,amsthm,amsfonts}
\usepackage{eurosym}
\usepackage{epsf}
\usepackage{epsfig}
\usepackage{color}

\newfont{\gothic}{eufm10}

\def\Z{{\mathbb{Z}}}                   \def\R{{\RR}}
\def\RR{{\mathbb{R}}}        \def\N{{\mathbb{N}}}        \def\Q{{\mathbb{Q}}}

        \newtheorem{theorem}{Theorem}[section]
\newtheorem{lemma}[theorem]{Lemma}
\newtheorem{proposition}[theorem]{Proposition}
\newtheorem{corollary}[theorem]{Corollary}
\newtheorem{definition1}[theorem]{Definition}

\newenvironment{definition}{\begin{definition1}\rm}{\end{definition1}}

\newtheorem{remark1}[theorem]{Remark}
\newenvironment{remark}{\begin{remark1}\rm}{\end{remark1}}

\newtheorem{example1}[theorem]{Example}

\def\barray{\begin{eqnarray*}}             \def\earray{\end{eqnarray*}}
\def\beq{\begin{equation}} \def\eeq{\end{equation}}

\makeatletter \title{\Large Continuity of the Peierls barrier and robustness of laminations}  
\author{Bla\v{z} Mramor\thanks{Institute of Mathematics, Albert-Ludwigs-Universit\"at, Freiburg, Germany, {\tt blazmramor@hotmail.com}.}  \ and Bob Rink\thanks{Department of Mathematics, VU University Amsterdam, The Netherlands, {\tt b.w.rink@vu.nl}.}\ . }
\begin{document}  \hyphenation{mini-mi-zers mini-mizer con-secu-tive}
\newcommand{\X}{\mathbb{X}}
\newcommand{\M}{\mathcal{M}}
\newcommand{\B}{\mathcal{B}}

\newcommand{\p}{\partial}
\maketitle
\noindent 

\abstract{We study the Peierls barrier $P_{\omega}(\xi)$ for a broad class of monotone variational problems. These problems arise naturally in solid state physics and from Hamiltonian twist maps. 

We start by deriving an estimate for the difference $\left| P_{\omega}(\xi) - P_{q/p}(\xi) \right|$ of the Peierls barriers of rotation numbers $\omega\in \R$ and $q/p\in \Q$. A similar estimate was obtained by Mather \cite{Matherpeierls} in the context of twist maps, but our proof is different and applies more generally. It follows from the estimate that $\omega\mapsto P_{\omega}(\xi)$ is continuous at irrational points. 

Moreover, we show that the Peierls barrier depends continuously on parameters and hence that the property that a monotone variational problem admits a lamination of minimizers of rotation number $\omega\in \R\backslash \Q$, is open in the $C^1$-topology.}

\section{Introduction}
We shall be interested in ``scalar monotone variational recurrence relations'' of the form
\begin{align}\label{recrelintro}
R(x_{i-r}, \ldots, x_{i+r}) =\sum_{j\in \Z} \p_{i}S(x_j, x_{j+1}, \ldots, x_{j+r-1}, x_{j+r}) = 0\ \mbox{for}\ x_i\in \R \ \mbox{and}\ i\in \Z\, .
\end{align}
Such problems are determined by a ``local potential'' $S:\R^{r+1}\to \R$ that describes the interaction between particles. The number $r\geq 1$ is the range of the interaction.  

In Section \ref{setting}, we specify conditions on the local potential $S$ that guarantee that (\ref{recrelintro}) is a ``monotone'' recurrence relation. These monotone problems arise in various contexts. For example in the study of Hamiltonian twist maps of the annulus, where $S=S(x_j, x_{j+1})$ is the generating function of the map, see \cite{gole01} or \cite{MatherForni}. Models with higher ranges of interaction appear in conservative lattice dynamics and solid state physics. A prototypical example to keep in mind is the generalized Frenkel-Kontorova problem
\begin{align}\nonumber
\sum_{k=1}^{r} a_k ( x_{i-k}-2x_i + x_{i+k}) = V'(x_i) \ \mbox{with}\ a_1,\ldots, a_r>0\ \mbox{and}\  V:\R\to\R\ \mbox{a periodic function}.
\end{align}
This equation describes the equilibrium states of a dislocation model in which particles experience a periodic background force and interact linearly with their neighbours. 
 
Problems of the form (\ref{recrelintro}) have been studied extensively in the scope of Aubry-Mather theory. One of the key results in this theory is that (\ref{recrelintro})  supports solutions of every rotation number $\omega\in \R$. They are the Birkhoff global minimizers. We shall denote these by 
$$\mathcal{M}_{\omega}:=\{ \, x:\Z\to\R\, |\, x \ \mbox{is a Birkhoff global minimizer of (\ref{recrelintro}) with rotation number} \ \omega\, \}\, .$$ 
When $\omega \in \R\backslash \Q$, then $\mathcal{M}_{\omega}$ is strictly ordered \cite{bangert87} and its recurrent subset is the well-known Aubry-Mather set of rotation number $\omega$. 
Following Moser \cite{moser86, moserbrasil, moser89}, we say that $\mathcal{M}_{\omega}$ is a {\it foliation} if for all $\xi\in \R$ there is an element $x\in \mathcal{M}_{\omega}$ that satisfies the initial condition $x_0=\xi$. Otherwise, one says that $\mathcal{M}_{\omega}$ forms a {\it  lamination}. In the setting of twist maps, foliations correspond to rotational invariant circles. These circles are important because they form the energy transport barriers for the map. In solid state physics models, foliations are associated to a ``sliding'' effect, and laminations to a ``pinning'' of particles. 

It is thus important to know whether $\mathcal{M}_{\omega}$ is a foliation or a lamination for a given local potential $S$. Or at least, setting a slightly more modest goal, one could ask what is the structure of $\mathcal{M}_{\omega}$ for a ``generic'' local potential. It turns out that the answer to this latter question is already quite delicate: it depends on the degree of irrationality of the rotation number. For example, in the case of Hamiltonian twist maps \cite{SalamonZehnder} and for some generalizations of the Frenkel-Kontorova model \cite{llavecalleja}, it follows from KAM theory that the set of local potentials for which $\mathcal{M}_\omega$ is a smooth foliation (and hence defines an invariant circle for the twist map), is open in the $C^k$-topology, under the conditions that $k$ is large enough and that the rotation number $\omega\in \R\backslash \Q$ is sufficiently irrational (e.g. Diophantine). 

On the other hand, a converse KAM theory for ``not very irrational'' rotation numbers has also been developed. For Hamiltonian twist maps it has for example been shown by Mather   \cite{Matherdestruction} that if the rotation number $\omega$ is rational or Liouville, then the set of generating functions $S(x_j, x_{j+1})$ for which $\mathcal{M}_{\omega}$ is a lamination (and hence does not define an invariant circle but for example a cantorus), is dense in the $C^{k}$-topology. In other words: if a twist map supports an invariant circle of a rational or Liouville rotation number, then this circle can be destroyed by an arbitrarily small perturbation of the generating function $S$. This destruction result was modified for the analytic setting by Forni \cite{Fornianalytic}, but the extension to general problems of the form (\ref{recrelintro}) has  been developed  only recently \cite{destruction} by the authors. 

The first goal of this paper is to develop a versatile framework for the study of foliations and laminations. In fact, to distinguish foliations from laminations, we shall make use of the Peierls barrier $P_{\omega}=P_{\omega}(\xi)$. This function can be thought of as a dislocation energy and it measures the ``minimal action'' of a Birkhoff sequence $x$ of rotation number $\omega$, under the constraint that $x_0=\xi$. In particular, it holds that $P_{\omega}\equiv 0$ if and only if $\mathcal{M}_{\omega}$ is a foliation. The first main result of this paper is the following technical statement on the Peierls barrier:

\begin{theorem}\label{th1}
Let $S$ be a local potential satisfying conditions {\bf A}-{\bf C} of Section \ref{setting} and let $L>0$. There exists a constant $C$ so that for all $\omega\in \R$ and $q/p\in \Q$ with $|\omega|, |q/p| \leq L$,
\begin{align}\label{estimateintro}
\left| P_{\omega}(\xi) - P_{q/p}(\xi) \right| \leq C\left( 1/|p| + |p\omega - q| \right) \ \mbox{uniformly in}\ \xi\, .
\end{align}
Consequently, the map $\omega\mapsto P_{\omega}(\xi)$ is continuous at irrational and H\"older continuous at Diophantine points. 
\end{theorem}

\noindent For Hamiltonian twist maps, precisely the same result was obtained by Mather \cite{Matherpeierls} and Theorem \ref{th1} generalizes his result to arbitrary recurrence relations of the form (\ref{recrelintro}). We should like to stress that Mather's derivation of (\ref{estimateintro}) relies on a specific property of recurrence relations that stem from a twist map: it is the fact that when $r=1$, the minimizers of (\ref{recrelintro}) can only cross once. This special crossing property does not hold for problems with a longer range of interaction though. Consequently, our extension of the results in \cite{Matherpeierls} requires a rather new approach that we present in this paper. We moreover conjecture that our approach can be extended to other settings in which Aubry-Mather theory has been applied, such as monotone problems on lattices \cite{llave-lattices, llave-valdinoci} and elliptic PDEs \cite{moser89, moser86}.

A byproduct of Theorem \ref{th1} is that $P_{\omega}(\xi)$ depends continuously on the local potential. As a consequence, it is an open property for a local potential to support a lamination of an irrational rotation number. This is the second main result of this paper:
\begin{theorem}\label{thm1intro}
Let $S$ be a local potential satisfying conditions {\bf A}-{\bf C} of Section \ref{setting} and let $\omega\in \R\backslash \Q$. Assume that $\mathcal{M}_\omega$ is a  lamination for $S$. Then there exists a $\delta>0$, such that for all local potentials $S^{\delta}$ with $\|S^{\delta}-S\|_{C^0}< \delta$ and $\|S^{\delta}-S\|_{C^1}\leq 1$, still $\mathcal{M}_{\omega}$ is a lamination.
\end{theorem}
\noindent 
Combined with the results in \cite{destruction}, Theorem \ref{thm1intro} implies that if $\omega$ is a Liouville number, then the set of local potentials for which $\mathcal{M}_{\omega}$ is a lamination, is not only dense but also open in the $C^{k}$-topology, for any $k\geq 2$.  All the other robustness results for laminations that the authors are aware of, apply only to Hamiltonian twist maps. Moreover, they are based on Green's criterium and require hyperbolicity of the corresponding twist map. See for example \cite{AubryMackayBaesens} and \cite{percival}.

\section{Aubry-Mather theory}\label{AM theory}

In this section we present some classical results from Aubry-Mather theory that we will use in this paper. The first results of Aubry-Mather theory were obtained for twist maps, independently by Aubry and Le Daeron \cite{AubryDaeron} and by Mather \cite{MatherTopology}. Many of these results remain true for more general monotone variational problems. All results in this section can be found in the existing literature and in particular in \cite{llave-lattices,llave-valdinoci, llave-valdinoci07,MramorRink1, destruction}.

\subsection{Monotone variational recurrence relations}\label{setting}
Let $r\geq 1$ be an integer and $S:\R^{r+1}\to \R$ a twice continuously differentiable function. We are interested in solutions $x:\Z\to\R$ of the recurrence relation
\begin{align}\label{recrel}
R(x_{i-r}, \ldots, x_{i+r}) := \sum_{j\in \Z} \p_{i}S(x_j, x_{j+1}, \ldots, x_{j+r-1}, x_{j+r}) = 0\ \mbox{for all}\ i\in \Z\, .
\end{align}
Here, we denoted $\p_{i}S(x_j, x_{j+1}, \ldots, x_{j+r-1}, x_{j+r}) := \frac{\p}{\p x_i} S(x_j, x_{j+1}, \ldots, x_{j+r-1}, x_{j+r})$. Let us remark that (\ref{recrel}) is a well-defined recurrence relation: the only nonzero terms in the sum occur for $j=i-r, \ldots, i$. We shall refer to the function $S$ as the {\it local interaction potential} of (\ref{recrel}) and to the integer $r$ as its {\it range of interaction}. We look for solutions to (\ref{recrelintro}) in $\R^{\Z}$, the space of bi-infinite sequences $x:\Z\to\R$.

Throughout this text we impose the following conditions {\bf A}-{\bf C} on the local potential:
\begin{itemize}
\item[{\bf A.}] Periodicity: $$S(x_j+1, \ldots, x_{j+r}+1)=S(x_j, \ldots, x_{j+r})\, .$$ 
This condition implies that the maps $x\mapsto S(x_{j}, \ldots, x_{j+r})$ descend to maps on $\R^{\Z}/\Z$.
\item[{\bf B.}] Monotonicity: $S$ is twice continuously differentiable and $$\partial_{i,k}S(x_j, \ldots, x_{j+r})  \leq 0 \ \mbox{for all} \ j \ \mbox{and all} \ i\neq k\ \mbox{and} \ \partial_{j,j+1}S(x_j, \ldots, x_{j+r})  < 0 \, .$$
This condition implies that (\ref{recrel}) is monotone, in the sense that
$ \p_k R(x_{i-r}, \ldots, x_{i+r}) \leq 0$ for all $k\neq i$. The condition is
 also called a {\it twist condition} or {\it ferromagnetic condition}. 
\item[{\bf C.}] Coercivity: $x\mapsto S(x_j, \ldots, x_{j+r})$ is bounded from below and there is a $j \leq k \leq j+r-1$ for which $$\lim_{|x_{k+1}-x_k|\to \infty}S(x_j, \ldots, x_{j+r})=\infty\, .$$
\end{itemize}
Conditions {\bf A}-{\bf C} are standard in Aubry-Mather theory. They guarantee that (\ref{recrel}) supports many interesting solutions. This will be explained below.

\subsection{Global minimizers}
One can think of the solutions to (\ref{recrel}) as the stationary points of the formal action 
\begin{align}\label{formalaction}
W(x):=\sum_{j\in \Z}S(x_j, \ldots, x_{j+r})\, .
\end{align}
Indeed, $R(x_{i-r}, \ldots, x_{i+r}) $ is precisely the (formal) derivative of $W(x)$ with respect to $x_i$.

Nevertheless, the formal sum (\ref{formalaction}) will in general be divergent. Thus, in order to make this formal variational principle useful, let us make a few definitions. Firstly,  for any finite interval $[i_0, i_1] \subset \Z$ we define a finite action
$W_{[i_0,i_1]}:\R^{\Z}\to \R$ by
$$W_{[i_0, i_1]}(x):=\sum_{j \in [i_0, i_1]} S(x_j, \ldots, x_{j+r})\, .$$
It is clear that $W_{[i_0, i_1]}(x)$ is a finite sum and is a function of only $x_{i_0}, \ldots, x_{i_1+r}$. On the other hand, when $i\in[i_0+r, i_1]$, then $\p_{i}S(x_j, \ldots, x_{j+r})=0$ for all $j< i_0$ and all $j> i_1$ and therefore, if $i\in[i_0+r, i_1]$,
$$\p_{i} W_{[i_0, i_1]}(x)= \sum_{j\in[i_0, i_1]} \p_{i}S(x_j, \ldots, x_{j+r}) = \sum_{j\in \Z} \p_{i}S(x_j, \ldots, x_{j+r}) = R(x_{i-r}, \ldots, x_{i+r}) \, .$$
With this in mind, we define a special type of solutions to (\ref{recrel}) as follows:

 \begin{definition}
A sequence $x: \Z\to\R$ is called a \textit{global minimizer}, if it holds for all finite intervals $[i_0, i_1]\subset \Z$ and for all $v: \Z\to \R$ with ${\rm supp}(v) \subset [i_0+r, i_1]$ that 
$$W_{[i_0, i_1]}(x) \leq W_{[i_0, i_1]}(x+v)\, .$$
\end{definition}
\noindent It is clear from the above considerations that global minimizers are solutions to (\ref{recrel}). 

We also remark for later reference:
\begin{proposition}\label{minimizerslimit}
The collection of global minimizers is closed under pointwise convergence.  \end{proposition}
\begin{proof} 
Let $x^1, x^2, \ldots$ be global minimizers and assume that the pointwise limit $x^{\infty} = \lim_{n\to \infty} x^n$ exists, that is $\lim_{n\to \infty} x^n_i = x^{\infty}_i$ for all $i\in \Z$. Moreover, let $v:\Z\to\R$ have finite support, say ${\rm supp}(v)\subset [i_0+r, i_1]$. Then it holds that  
$$W_{[i_0,i_1]}(x^n) \leq W_{[i_0,i_1]}(x^n+v)\ \mbox{for all}\ n=1,2,\ldots\, .$$
Using that $W_{[i_0,i_1]}$ is continuous for pointwise convergence and taking the limit for $n\to\infty$, it then follows that also 
$$W_{[i_0,i_1]}(x^{\infty}) \leq W_{[i_0,i_1]}(x^{\infty}+v)\, .$$ 
Hence also $x^{\infty}$ is a global minimizer.
\end{proof}
\subsection{The Birkhoff property}
It turns out that many global minimizers have the so-called {\it Birkhoff property}. To define this property, let us introduce the translates of a sequence as follows:

\begin{definition}
For $k,l \in \Z$ and $x\in \R^{\Z}$, we define the translate $\tau_{k,l}x\in \R^{\Z}$ by
\begin{align} \label{translations}
(\tau_{k,l}x)_i := x_{i-k}+l\, .
\end{align}
\end{definition}
\noindent Note that condition {\bf A} ensures that if $x$ is a global minimizer of (\ref{recrel}), then also $\tau_{k,l}x$ is a global minimizer. In other words: the collection of global minimizers is translation-invariant. 

We also introduce a partial ordering on $\R^\Z$:

\begin{definition} We write 
\begin{itemize} 
	\item $x\leq y$ if $x_i\leq y_i$ for all $i\in \Z$ (ordering),
	\item $x<y$ if  $x_i\leq y_i$ for all $i\in \Z$ but $x\neq y$ (weak ordering),
	\item $x\ll y$ if $x_i<y_i$ for all $i\in \Z$ (strict ordering).
\end{itemize} 
\end{definition}
\noindent Now we can define Birkhoff sequences: 
\begin{definition}\label{birkhoff}
We call a sequence $x\in \R^\Z$ {\it Birkhoff} if the collection $$\{\tau_{k,l}x\, |\, k,l\in \Z\} \ \text{is ordered. }  $$
We denote the set of Birkhoff sequences by $\mathcal{B}$.
\end{definition} 
\noindent Thus, a Birkhoff sequence is a sequence that does not ``cross'' any of its integer translates. A nice and complete overview of the properties of Birkhoff sequences can be found in \cite{gole01}. A noteworthy one is that they have a rotation number:
\begin{proposition}\label{proprotnr}         
If $x\in \R^{\Z}$ is Birkhoff, then there is a unique $\omega=\omega(x)\in \R$ such that 
\begin{align}\label{birkhoffest}
|x_i-(x_0+ \omega \cdot i) |\leq 1\, .
\end{align}
In particular, $\lim_{n\to \pm \infty}\frac{x_n}{n} = \omega(x)$. Moreover, the {\rm rotation number} $\omega(x)$ depends continuously on the sequence: when $x^1, x^2, \ldots, x^{\infty}$ are Birkhoff sequences of rotation numbers $\omega_1, \omega_2, \ldots, \omega_{\infty}$ and $\lim_{n\to \infty} x^n = x^{\infty}$ pointwise, then $\lim_{n\to \infty} \omega_n = \omega_{\infty}$.
\end{proposition}
\noindent See \cite{MramorRink1} (Lemma 3.5) for a proof of this proposition. We write $\mathcal{B}_{\omega}:=\{ x\in \mathcal{B}\, |\, \omega(x)=\omega\, \}$. 
It is clear that $\mathcal{B}_{\omega}$ is closed in the topology of pointwise convergence and as an important corollary of Proposition \ref{proprotnr} we also mention:
\begin{corollary}
Let $L>0$. The collection $\bigcup_{|\omega|\leq L} \mathcal{B}_{\omega}/\Z$ is compact in the topology of pointwise convergence.
\end{corollary}
\begin{proof}
Let $x^1, x^2, \ldots \in \bigcup_{|\omega| \leq L} \mathcal{B}_{\omega}$. Then it follows from (\ref{birkhoffest}) that $|x^n_{i+1}-x^n_i|\leq K$ for $K:=L+2$. By equivalence, we may moreover assume that $x^{n}_0\in [0,1]$ and  therefore $x^{n}_i \in \left[-K|i|,  1+K |i| \right]$. Hence, by Tychonov's theorem there is a subsequence $x^{n_j}$ that limits pointwise to a sequence $x^{\infty}\in \R^\Z$ with $x^{\infty}_0\in [0,1]$. It is clear $x^{\infty}$ is Birkhoff, because $\mathcal{B}$ is closed, and it follows from Proposition \ref{proprotnr} that $\omega(x^{\infty}) = \lim_{j\to\infty} \omega(x^{n_j})\in [-L,L]$. This proves the corollary.
\end{proof}
\noindent Finally, the following technical result is well-known and will be used a few times in this paper. We refer to \cite{MramorRink1} (Proposition 3.8) for a proof. 
\begin{proposition}\label{numbertheory}
Let $x$ be Birkhoff of rotation number $\omega$. Then 
$$ \tau_{k,l}x>x\ \mbox{if}\ -k\omega + l >0\  \mbox{and} \ \tau_{k,l}x<x \ \mbox{if} \ -k\omega + l <0 \, .$$
\end{proposition}
 \subsection{Periodic minimizers}\label{per section}
For $p,q\in \Z$, let us define the collection of $(p,q)$-periodic sequences as 
$$\X_{p,q}:=\{x\in \R^\Z \ | \ \tau_{p,q}x=x\}\, .$$
Elements of $\X_{p,q}$ have rotation number $q/p$ (also if they are not Birkhoff).

An important remark is that a $(p,q)$-periodic sequence $x\in \X_{p,q}$ is a solution to (\ref{recrel}) if and only if it is a stationary point of the periodic action function
$$W_{p,q}:\X_{p,q}\to\R\ \mbox{defined by}\ W_{p,q}(x):= W_{[0,p-1]}(x) = \sum_{j=0}^{p-1} S(x_j, \ldots, x_{j+r})\, .$$
Because $\X_{p,q}$ is finite-dimensional and $W_{p,q}(x)$ is a finite sum, these stationary points are well-defined and in particular we call $x\in \X_{p,q}$ a {\it periodic minimizer} or {\it $(p,q)$-minimizer} if  
$$W_{p,q}(x) \leq W_{p,q}(y) \ \mbox{for all}\ y\in \X_{p,q}\, .$$ 
\noindent Condition {\bf C} guarantees that periodic minimizers of all periods exist, see also \cite{MramorRink1} (Theorem 4.3). The following proposition summarizes what we need to know about periodic minimizers. For a full proof of this proposition, we refer to \cite{gole01} or \cite{MramorRink1} (Section 4.1). 
\begin{proposition}
\label{periodicminimizersproposition}
For all $p,q\in \Z$ with $p\neq 0$, the collection 
$$\mathcal{M}_{p,q} := \{x\in \X_{p,q}\, |\, x \ \mbox{is a}\ (p,q)\mbox{-minimizer}\, \}$$
is nonempty, closed under pointwise convergence, translation-invariant and strictly ordered. In particular, every element of $\mathcal{M}_{p,q}$ is Birkhoff. Moreover, $x\in \X_{p,q}$ is a $(p,q)$-minimizer if and only if it is an $(np, nq)$-minimizer (for any $n\in \N$), if and only if it is a global minimizer.
\end{proposition}
 \noindent We shall denote by $\mathcal{B}_{p,q}:= \mathcal{B} \cap \X_{p,q} \subset \mathcal{B}_{q/p}$ the set of $(p,q)$-periodic Birkhoff sequences. Proposition \ref{periodicminimizersproposition} shows that $\mathcal{M}_{p,q}\subset \mathcal{B}_{p,q}$ and that $\mathcal{M}_{p,q} = \mathcal{M}_{np,nq}$. In fact, it even holds that $\B_{np,nq}=\B_{p,q}$, see \cite{MramorRink1} (Theorem 3.12).
 
The part of Proposition \ref{periodicminimizersproposition} that says that every periodic minimizer has the Birkhoff property is also known as {\it Aubry's lemma} or the {\it non-crossing lemma}. Proofs of Aubry's lemma can also be found in \cite{AubryDaeron}, \cite{llave-lattices} and \cite{MatherForni}.  

\subsection{Quasi-periodic minimizers}\label{qper section}

Nonperiodic global minimizers need not have the Birkhoff property, see \cite{dichotomy}. Nevertheless, Birkhoff global minimizers of irrational rotation numbers exist and they can be constructed as limits of periodic minimizers. This is the content of Theorem \ref{existenceqp} below. In the context of twist maps, this theorem 
  is originally due to Aubry and le Daeron \cite{AubryDaeron} and Mather \cite{MatherTopology}. Generalizations to finite range variational recurrence relations have been made by Angenent \cite{angenent90} and in the context of certain lattice problems, Theorem \ref{existenceqp} has been formulated for the first time by Koch, de la Llave and Radin in \cite{llave-lattices}.  We sketch the proof.

\begin{theorem}\label{existenceqp}
For every $\omega\in \R\backslash \Q$ the collection $$\mathcal{M}_{\omega}:= \{ x\in \mathcal{B}_{\omega}\, |\, x \ \mbox{is a global minimizer of (\ref{recrel})}\, \}$$ is nonempty, closed under pointwise convergence, translation-invariant and strictly ordered.
\end{theorem}
\begin{proof}\ [{\bf Sketch}]  
Given $\omega\in\R \backslash \Q$, choose a sequence $q_n/p_n\in \Q$ with $\lim_{n\to\infty} q_n/p_n=\omega$. Let $x^{n}\in \mathcal{M}_{p_n, q_n}$ be a corresponding sequence of periodic minimizers. We have seen that these exist and have rotation number $q_n/p_n$. Each of them is a Birkhoff global minimizer and by translation-invariance it can be assumed that $x^{n}_0\in [0,1]$. Because there is a constant $L>0$ so that $|q_n/p_n| \leq  L$, there exists by compactness a subsequence $x^{n_j}$ that converges to a Birkhoff global minimizer $x^{\infty}$. Proposition \ref{proprotnr} guarantees that $x^{\infty}$ has rotation number $\omega$. This proves that $\mathcal{M}_{\omega}$ is nonempty. 

Proposition \ref{minimizerslimit} implies that $\mathcal{M}_{\omega}$ is closed under pointwise convergence and it is obvious that it is translation-invariant. The proof that $\mathcal{M}_{\omega}$ is strictly ordered if $\omega\in\R\backslash \Q$ is due to Bangert \cite{bangert87}. He originally gave this proof in the context of the study of minimal solutions of variational elliptic PDEs on the torus, as studied also by Moser in \cite{moser86, moser88, moserbrasil, moser89}. His proof is nontrivial and we do not provide it here. 
\end{proof}

\noindent One defines the {\it Aubry-Mather set} $\mathcal{M}_{\omega}^{\rm rec}\subset \mathcal{M}_{\omega}$ of rotation number $\omega\in \R\backslash \Q$ to be the recurrent subset of $\mathcal{M}_{\omega}$. This means that $\mathcal{M}_{\omega}^{\rm rec}$ is the unique smallest nonempty, closed and translation-invariant subset of $\mathcal{M}_{\omega}$. It was shown by Bangert \cite{bangert87} that $\mathcal{M}_{\omega}^{\rm rec}$ is well-defined. 

	

It is well-known, see for example \cite{MramorRink1} (Theorem 4.18) that $\mathcal{M}^{\rm rec}_{\omega}$ is either topologically connected or a Cantor set. In case that $\M^{\rm rec}_\omega$ is topologically connected, then it obviously holds that $\M^{\rm rec}_\omega=\M_\omega$. In particular, $\M_\omega$ is then a foliation. If $\M^{\rm rec}_\omega$ is a Cantor set, then $\M_\omega$ could still be a minimal foliation: this happens when the gaps in $\mathcal{M}_{\omega}^{\rm rec}$ are filled by minimizers. If, however, there is a gap in $\M^{\rm rec}_\omega$ that is not filled by minimizers, then $\M_\omega$ is a lamination.

\subsection{A Lipschitz estimate}
We finish this introductory section by providing a Lipschitz estimate that we will use extensively when we compare the actions of two distinct sequences.
\begin{proposition}\label{lipschitz}
Let $K>0$ be a constant and define 
$$\X_{K}:= \{ x: \Z\to \R\, |\, |x_{i+1}-x_i|\leq K\ \mbox{for all}\ i\in\Z\, \} \subset \R^{\Z}\, .$$ 
There exists a constant $D>0$ so that for all $x,y\in \X_{K}$ and for any $i_0\leq i_1$, 
\begin{align}\label{energyestimate}
|W_{[i_0,i_1]}(x)-W_{[i_0,i_1]}(y)|\leq D \!\!\! \sum_{i\in [i_0, i_1+r]}|x_i-y_i|\, .
\end{align}
\end{proposition}

\begin{proof} 
By property {\bf A}, the local potential $x\mapsto S(x_0, \ldots, x_{r})$ defines a function on $\R^{\Z}/\Z$. Moreover, the subset $\X_{K}/ \Z \subset \R^{\Z}/ \Z$ is compact for the topology of pointwise convergence. Hence, because $S$ was assumed continuously differentiable, there is a constant $c>0$ so that 
$$|\p_{k}S(x_0, \ldots, x_{r}) |\leq c\ \mbox{for all}\ k\in \Z \ \mbox{and all} \ x\in \X_K\, .$$
On the other hand, if we define the sequence $z:\Z\to\R$ by $z_i:=x_{i+j}$, then clearly $z\in \X_{K}$ if $x\in \X_{K}$ and $\p_{k}S(x_j, \ldots, x_{j+r}) = \p_{k-j}S(z_0, \ldots, z_{r})$.  As a result,  
$$|\p_{k}S(x_j, \ldots, x_{j+r}) |\leq c\ \mbox{uniformly in}\ j \ \mbox{and}\ k \ \mbox{and for all}\ x\in \X_K\, .$$
The desired Lipschitz estimate now follows from interpolation. Indeed, for $x,y\in \X_{K}$, 
\begin{align}\nonumber
& |W_{[i_0,i_1]}(x)-W_{[i_0,i_1]}(y)| \leq \sum_{j\in [i_0,i_1]} \left| S(x_j, \ldots, x_{j+r})-S(y_j, \ldots, y_{j+r}) \right| = \\ \nonumber
& \sum_{j\in [i_0,i_1]} \left| \int_0^1\frac{d}{d\tau} S(\tau x_j+(1-\tau)y_j, \ldots, \tau x_{j+r}+(1-\tau)y_{j+r} )d\tau \right| \leq \\ \nonumber  & 
 \sum_{j\in [i_0,i_1]} \sum_{k=j}^{j+r} \left( \int_0^1\left| \p_kS(\tau x_j+(1-\tau)y_j, \ldots, \tau x_{j+r}+(1-\tau)y_{j+r} )\right| d\tau\right)  |x_k-y_k| \leq \\ \nonumber  &  
\sum_{j\in [i_0, i_1]} \sum_{k=j}^{j+r} c\, |x_k-y_k| \leq  c\, (r+1)\!\!\!\sum_{i\in [i_0, i_1+r]} |x_i-y_i|\ .
\end{align}
Note: we used that $\X_{K}$ is convex. This proves the proposition if we choose $D:=c\, (r+1)$ and $c=||S||_{C^1(\X_K)}$.
\end{proof}
\noindent It follows from (\ref{birkhoffest}) that $\mathcal{B}_{\omega}\subset \X_{K}$ if $K\geq 2+|\omega|$. Nevertheless, the set $\X_K$ is quite a bit larger than $\B_{\omega}$. It will be convenient to have inequality (\ref{energyestimate}) on the whole of $\X_{K}$, to produce energy estimates for sequences that are ``not too far from Birkhoff''.



\section{The Peierls barrier}
In this section, we introduce the Peierls barrier. It is a tool to distinguish foliations from laminations. Most importantly, it determines whether for a given $\xi\in \R$ there exists a global minimizer $x\in \mathcal{M}_{\omega}$ that satisfies the initial condition $x_0=\xi$. 
\subsection{Constrained minimizers}\label{constrainedsection}
We start this section by introducing a type of constrained minimizers:
\begin{definition}
For $\xi\in \R$, a sequence $x: \Z\to\R$ is called a global $\xi$-\textit{minimizer}, if $x_0=\xi$ and if it holds for all finite intervals $[i_0, i_1]\subset \Z$ and for all $v:\Z \to \R$ with ${\rm supp}(v) \subset [i_0+r, i_1]$ and $v_0=0$ that $$W_{[i_0, i_1]}(x) \leq W_{[i_0, i_1]}(x+v)\, .$$
\end{definition}
\noindent Not surprisingly, global $\xi$-minimizers in general need not be solutions to (\ref{recrel}). 

By an obvious analogue of Proposition \ref{minimizerslimit}, the collection of global $\xi$-minimizers is closed under pointwise convergence. In particular, one can hope to construct quasi-periodic global $\xi$-minimizers as limits of periodic $\xi$-minimizers:
\begin{definition}
For $\xi\in \R$ and $p,q\in \Z$, a periodic sequence $x\in \X_{p,q}$ is called a {\it periodic} $\xi$-\textit{minimizer} if $x_0=\xi$ and if
$$W_{p,q}(x) \leq W_{p,q}(y)\ \mbox{for all}\ y\in \X_{p,q}\ \mbox{with} \ y_0=\xi\, .$$
We denote the collection of $(p,q)$-periodic $\xi$-minimizers by 
$$\mathcal{M}_{p,q}(\xi):= \{x\in \X_{p,q}\, |\, x \ \mbox{is a periodic}\ \xi\mbox{-minimizer}\, \}\, .$$
\end{definition}
\noindent It follows from condition {\bf C} that periodic $\xi$-minimizers of all periods exist for all $\xi\in\R$. It also turns out that periodic $\xi$-minimizers satisfy a (weak) version of the Aubry lemma. 

\begin{lemma}\label{variant AL}
Let $p,q\in \Z$ be relative prime, let $\xi\in \R$ and let $x \in \mathcal{M}_{p,q}(\xi)$ be a periodic $\xi$-minimizer. Then there are periodic minimizers $x^{-}, x^{+}\in \M_{p,q}$ such that 
$$x^{-} \leq x \leq x^{+}$$
and for which there is no $y\in \mathcal{M}_{p,q}$ with $x^{-} < y < x^{+}$.  As a consequence, $\mathcal{M}_{p,q}(\xi) \subset \mathcal{B}_{p,q}$ and every $x\in \mathcal{M}_{p,q}(\xi)$ is a global $\xi$-minimizer.
\end{lemma}

\begin{proof}
Because $\mathcal{M}_{p,q}$ is strictly ordered, we can define 
$$x^{-}:= \sup\{ y\in \mathcal{M}_{p,q}\, |\, y_0 \leq \xi\, \}\ \mbox{and} \ x^{+}:= \inf\{ y\in \mathcal{M}_{p,q}\, |\, y_0 \geq \xi\, \}\, .$$ 
Then $x^{\pm}\in \mathcal{M}_{p,q,}$ and there is no $y\in \M_{p,q}$ with $x^{-} <y < x^{+}$.

To prove that $x^{-} < x < x^{+}$, we remark that condition {\bf B} implies that the periodic action satisfies a minimum-maximum principle. Indeed, if we define $m,M\in \X_{p,q}$ as $$m_i:=\min\{x_i,x^{-}_i\} \ \mbox{and}\ M_i:=\max\{x_i,x^{-}_i\}\, ,$$ 
 then it holds that 
\begin{align}\label{minmax}
W_{p,q}(M)+W_{p,q}(m) \leq W_{p,q}(x)+W_{p,q}(x^{-})\, .
\end{align}
For a proof of this estimate, see \cite{MramorRink1} (Lemma 4.4). Because $x$ is minimal subject to $\xi$, and because $M_0=\xi$, it must hold that $W_{p,q}(M)\geq W_{p,q}(x)$ and hence by (\ref{minmax})  that $W_{p,q}(m)\leq W_{p,q}(x^{-})$. This implies that $m\in \M_{p,q}$ and because $m_0=x^{-}_0$ that $m=x^{-}$. Hence, $x^{-} \leq x$. A similar deduction implies that $x\leq x^{+}$. 

To prove that $x$ is Birkhoff, let $k,l\in \Z$. If $-k(q/p)+l=0$ then $k=np$ and $l=nq$ because $p$ and $q$ were assumed relative prime. This implies that $\tau_{k,l}x=x$ because $x\in \X_{p,q}$. Otherwise, if for example $-k(q/p)+l>0$, then by Proposition \ref{numbertheory} it must hold that $\tau_{k,l}x^{-} > x^{-}$. But then it follows that $\tau_{k,l}x^{-} \geq x^{+}$ and hence also that  $\tau_{k,l}x \geq \tau_{k,l}x^{-} \geq x^{+}\geq x$. A similar argument in case that $-k(q/p)+l<0$ proves that $x$ is Birkhoff.

Finally, the argument showing that every periodic $\xi$-minimizer is a global $\xi$-minimizer is identical to the proof that every periodic minimizer is a global minimizer. One can copy  this argument verbatim from Lemma 4.7 and Theorem 4.8 in \cite{MramorRink1}. 
\end{proof}

\noindent As expected, Lemma \ref{variant AL} implies that nonperiodic global $\xi$-minimizers can be constructed as limits of periodic $\xi$-minimizers. More precisely, when $\omega\in \R\backslash \Q$ and $q_n/p_n\to\omega$, then one may select periodic $\xi$-minimizers $x^{p_n,q_n}(\xi)\in \mathcal{M}_{p_n, q_n}(\xi)$ and periodic minimizers $x^{p_n,q_n,\pm} \in \mathcal{M}_{p_n,q_n}$ with the properties described in Lemma \ref{variant AL}. By compactness these can be chosen (by passing three times to a subsequence if necessary) in such a way that $x^{p_n,q_n,\pm}\to x^{\omega, \pm}\in \mathcal{M}_{\omega}$ and $x^{p_n,q_n}(\xi)\to x^{\omega}(\xi)$. Then $x^{\omega}(\xi)\in \mathcal{B}_{\omega}$ is a global $\xi$-minimizer and we have the inequalities
$$x^{\omega,-}\leq x^{\omega}(\xi) \leq x^{\omega,+}\, .$$
We finish this section with a technical result about these quasi-periodic minimizers that we will need below. It says that the gap $$[x^{\omega,-}, x^{\omega,+}]:= \{\, x\in \R^{\Z}\, |\, x^{\omega,-} \leq x \leq x^{\omega,+}\, \}$$
between $x^{\omega,-}$ and $x^{\omega,+}$  is bounded in $l_1(\Z)$.
\begin{lemma}
It holds that 
$$\sum_{j\in \Z} \left| x^{\omega,+}_j - x^{\omega,-}_j  \right| \leq 1\, .$$
 \end{lemma}
\begin{proof}
Whenever $p_n, q_n$ are relative prime and $x^{p_n,q_n,-}\leq x^{p_n,q_n,+}$ are elements of $\mathcal{M}_{p_n,q_n}$ for which there is no $y\in \mathcal{M}_{p_n,q_n}$ with $x^{p_n,q_n,-} < y < x^{p_n,q_n,+}$, then
 \begin{align}\label{lessthanoneperiodic}
 \sum_{j=1}^{p_n} \left| x^{p_n, q_n,+}_j - x^{p_n, q_n,-}_j \right| \leq 1\, .
\end{align} 
For a proof of this standard fact, based on the pigeonhole principle, see for example \cite{MramorRink1} (Theorem 10.2). 


The same estimate remains true in the limit: if $x^{p_n, q_n,\pm}\to x^{\omega, \pm}$ pointwise, then for all $M>0$ and all $\varepsilon>0$ there is an $N>0$ so that for all $-M\leq j\leq M$ and $n\geq N$ it holds that $|x^{\omega, \pm}_j - x^{p_n, q_n,\pm}_j| < \varepsilon/ 2(2M+1)$. As a consequence, when $2M+1\leq p_n$,
$$\sum_{j=-M}^M \left| x^{\omega, +}_j-x^{\omega, -}_j \right| \leq \sum_{j=-M}^M  |x^{\omega, +}_j - x^{p_n,q_n, +}_j|+ |x^{p_n,q_n, +}_j - x^{p_n,q_n, -}_j| + | x^{p_n,q_n, -}_j-x^{\omega, -}_j |  \leq 1+\varepsilon\, .$$
Since this is true for all $M$ and $\varepsilon$, the lemma follows.
\end{proof}

\subsection{Definition and properties of the Peierls barrier}
To distinguish $\xi$-minimizers from minimizers, we now introduce the Peierls barrier. It compares the action of a $\xi$-minimizer to the action of an unconstrained minimizer. The periodic Peierls barrier $P_{q/p}(\xi)$ is most easily defined, namely as follows:

\begin{definition}
For $\xi\in\R$ and $q/p\in \Q$ a rational in lowest terms, we define the periodic Peierls barrier $P_{q/p}(\xi)$ as
$$P_{q/p}(\xi) :=  \!\!\!\! \min_{\tiny \begin{array}{c} x\in \X_{p,q} \\ x_0=\xi \end{array}}  \!\!\!\! W_{p,q}(x) - \!\!\!\! \min_{\tiny \begin{array}{c} x\in \X_{p,q} \end{array}}  \!\!\!\! W_{p,q}(x) \geq 0\, . 
 $$
\end{definition}
\noindent For the convenience of the reader, we state the following as a proposition:
\begin{proposition}\label{obviousprop}
Let $q/p\in\Q$ be a rational in lowest terms and $\xi\in \R$. There exists a periodic minimizer $x\in \mathcal{M}_{p,q}$ with $x_0=\xi$ if and only if $P_{q/p}(\xi)=0$. In particular, $\mathcal{M}_{p,q}$ is a foliation if and only if $P_{q/p}\equiv 0$.
\end{proposition}
\begin{proof}
Obvious from the definition of the periodic Peierls barrier.
\end{proof} 
\begin{remark}
We will not define the ``asymptotically periodic'' Peierls barrier functions $P_{(q/p)^+}(\xi)$ and $P_{(q/p)^-}(\xi)$. Unlike in \cite{Matherpeierls}, our arguments do not make use of these quantities.
\end{remark}
\noindent Next, we introduce the quasi-periodic Peierls barrier:
\begin{definition}\label{qppeierlsdef2}
For $\xi\in \R$ and $\omega\in\R\backslash \Q$ we define the quasi-periodic Peierls barrier as
\begin{align}\label{defirratpeierls2}
P_{\omega}(\xi) := \lim_{q/p\to \omega} P_{q/p}(\xi)\, . 
\end{align}
\end{definition}
\noindent Of course, it is not at all clear a priori that the limit $\lim_{q/p\to \omega} P_{q/p}(\xi)$ is well-defined. This fact will be proved in Section \ref{modulussection}. 
\begin{remark}
In \cite{Matherpeierls} the quasi-periodic Peierls barrier is defined differently, namely as 
\begin{align}\label{matherdef}
P_{\omega}(\xi) := \!\!\!\! \min_{\tiny \begin{array}{c} y^{\omega,-}\leq x \leq y^{\omega,+} \\ x_0=\xi \end{array}} \ \sum_{j\in \Z}  \left( S(x_j, \ldots, x_{j+r}) - S(y_j^{\omega, -}, \ldots, y_{j+r}^{\omega,-}) \right)\, .
\end{align}
Here, $y^{\omega, \pm}$ are the two nearest recurrent minimizers to $\xi$, that is
$$y^{\omega,-} :=\sup \{z\in \mathcal{M}_{\omega}^{\rm rec} \, | \, z_0 \leq \xi\, \}\ \mbox{and}\ y^{\omega,+} :=\inf \{z\in \mathcal{M}_{\omega}^{\rm rec} \, | \, z_0 \geq \xi\, \}\, .$$ 
Another reasonable option would have been to define $P_{\omega}(\xi)$ by a formula identical to (\ref{matherdef}), but with $y^{\omega, \pm}$ replaced by the global minimizers $x^{\omega, \pm}$ that were constructed in Section \ref{constrainedsection}. 

Nevertheless, we found Definition \ref{qppeierlsdef2} by far the most convenient one to work with. In fact, it can be proved that all these definitions yield the same value for $P_{\omega}(\xi)$. \hfill $\triangle$
\end{remark}
\noindent For the remainder of this section, we will simply assume that the quasi-periodic Peierls barrier is well-defined by (\ref{defirratpeierls2}), so that for any sequence of rationals $q_n/p_n$ in lowest terms that converges to $\omega$, the limit $\lim_{n\to\infty} P_{q_n/p_n}(\xi)$ exists and is independent of the chosen sequence of rationals. Then we can prove the following analogue of Proposition \ref{obviousprop}:

\begin{theorem}\label{foliationtheorem}
Let $\omega\in \R\backslash\Q$ and $\xi\in \R$. There exists a global minimizer $x\in \mathcal{M}_{\omega}$ with $x_0=\xi$ if and only if $P_{\omega}(\xi)=0$. In particular, $\mathcal{M}_{\omega}$ is a foliation if and only if $P_{\omega}\equiv 0$.
\end{theorem}
\begin{proof}
Recall from Section \ref{constrainedsection} that there exists a particular sequence of rationals $q_n/p_n\to\omega$ for which there are periodic minimizers $x^{n,\pm}:=x^{p_n,q_n, \pm}\in \mathcal{M}_{p_n,q_n}$ and periodic $\xi$-minimizers $x^{n}(\xi):= x^{p_n,q_n}(\xi)\in \mathcal{M}_{p_n,q_n}(\xi)$ such that $x^{n,-}\leq x^n(\xi)\leq x^{n,+}$ and as $n\to\infty$,
$$x^{n,-} \to x^{\omega, -}\in \mathcal{M}_{\omega} \, , \ x^{n,+}\to x^{\omega, +}\in \mathcal{M}_{\omega} \ \mbox{and}\ x^{n}(\xi) \to x^{\omega}(\xi)\in\mathcal{B}_{\omega}\, .$$
Here, $x^{\omega}(\xi)$ is a global $\xi$-minimizer. It holds that  $x^{\omega,-}\leq x^{\omega}(\xi)\leq x^{\omega,+}$ and 
\begin{align}\label{l1again}
\sum_{j\in \Z} \left| x^{\omega,+}_j - x^{\omega,-}_j  \right| \leq 1\, .
\end{align}
{\bf Proof of $\Leftarrow$}: Assume for a start that $x^{\omega}(\xi)\notin\mathcal{M}_{\omega}$. Then there exists a finite support variation $v:\Z\to\R$, say with ${\rm supp}(v)\subset [i_0+r, i_1]$, such that  
$$W_{[i_0, i_1]}(x^{\omega}(\xi)) \geq W_{[i_0, i_1]}(x^{\omega}(\xi)+v) +\varepsilon$$
for some $\varepsilon>0$. Obviously $0\in {\rm supp}(v)$ because $x^{\omega}(\xi)$ is a global $\xi$-minimizer. 

Now we define, for $n$ so large that $p_n \geq i_1-i_0+r+1$, the sequences $X^n\in \X_{p_n,q_n}$ by
$$(X^n)_j := (x^{n}(\xi)+v)_j \ \mbox{for}\  j\in [i_0,i_0+p_n -1]\, .$$
 Then it holds that $X_j^n=(x^{n}(\xi)+v)_j$ for $j\in [i_0, i_1+r]$, so that $$W_{[i_0, i_1]}(X^n) = W_{[i_0, i_1]}(x^n(\xi)+v)\, .$$ 
 Moreover, $X^n_j=(x^n(\xi))_j$ for all $j\in [i_1+1, i_0+p_n-1+r]$ (because ${\rm supp}(v) \subset[i_0+r,i_1]$) and therefore $$W_{[i_1+1, i_0+p_n-1]}(X^n) = W_{[i_1+1, i_0+p_n-1]}(x^n(\xi))\, .$$  
 As a result, 
 \begin{align}
& W_{p_n,q_n}(x^{n}(\xi)) - W_{p_n,q_n}(X^{n}) = W_{[i_0, i_0+p_n-1]}(x^{n}(\xi)) - W_{[i_0, i_0+p_n-1]}(X^{n}) = \nonumber \\ 
 & W_{[i_0, i_1]}(x^{n}(\xi)) - W_{[i_0, i_1]}(X^{n}) + W_{[i_1+1, i_0+p_n-1]}(x^{n}(\xi)) - W_{[i_1+1, i_0+p_n-1]}(X^{n}) = \nonumber \\
& W_{[i_0, i_1]}(x^{n}(\xi)) - W_{[i_0, i_1]}(x^{n}(\xi)+v) \xrightarrow{n\to\infty} 
W_{[i_0, i_1]}(x^{\omega}(\xi))  - W_{[i_0, i_1]}(x^{\omega}(\xi)+v) \geq \varepsilon\, . \nonumber
\end{align}
But obviously $$P_{q_n/p_n}(\xi) \geq W_{p_n,q_n}(x^n(\xi)) - W_{p_n,q_n}(X^n) $$
and we conclude that 
$$\lim_{n\to\infty} P_{q_n/p_n}(\xi) \geq \varepsilon>0\, .$$
This proves that if $P_{\omega}(\xi)=\lim_{n\to\infty} P_{q_n/p_n}(\xi)=0$, then there must be a global minimizer $x\in\mathcal{M}_{\omega}$ with $x_0=\xi$ (namely $x^{\omega}(\xi))$.
\\ \mbox{} \\
{\bf Proof of $\Rightarrow$}:
Next, let us assume that there exists a global minimizer $x\in \mathcal{M}_{\omega}$ with $x_0=\xi$. Because $\mathcal{M}_{\omega}$ is strictly ordered (recall the result by Bangert), it must hold that
$$x^{\omega,-}\leq x\leq x^{\omega,+}\, .$$
Let us argue by contradiction now and assume that 
$P_{\omega}(\xi)=\lim_{n\to\infty} P_{q_n/p_n}(\xi)=\varepsilon>0$.
To show that this is impossible, we choose an integer $M\geq r$ so large that 
$$\left| x^{\omega,+}_j - x^{\omega,-}_j  \right| \leq \varepsilon / 8Dr\ \mbox{for all}\ j\in [-M,-M+r-1] \cup [M+1,M+r]\, .$$
Here, $D$ is as in Proposition \ref{lipschitz}, with $K$ chosen in such a way that $\mathcal{B}_{\omega}\subset \X_K$. Such an $M$ exists because of (\ref{l1again}). 
Now we define
\begin{align}\nonumber 
X_j & :=  \left\{ \begin{array}{ll} x_j & \mbox{if}\ j\in [-M+r,M] \\ x^{\omega,-}_j &\mbox{otherwise}  \end{array}\right. \\ \nonumber 
Y_j & := \left\{ \begin{array}{ll} x^{\omega,-}_j & \mbox{if}\ j\in [-M+r,M] \\ x_j & \mbox{otherwise} \end{array}\right.
\end{align}
It is clear that $X$ is a variation of $x^{\omega,-}$ supported in $[-M+r,M]$ and $Y$ is a variation of $x$ supported in $[-M+r,M]$. On the other hand, by construction $X$ is extremely close to $x$ and $Y$ to $x^{\omega,-}$, in the sense that 
\begin{align}\nonumber 
\sum_{j\in [-M,M+r]} |X_j-x_j| \leq \varepsilon / 4D\ \ \mbox{and}\!\!\! \sum_{j\in [-M,M+r]} |Y_j-x^{\omega,-}_j| \leq \varepsilon / 4D \, .
\end{align}
In particular, it follows from Proposition \ref{lipschitz} that
$$|W_{[-M,M]}(X)-W_{[-M,M]}(x)| \leq \varepsilon/4\ \mbox{and}\ |W_{[-M,M]}(Y)-W_{[-M,M]}(x^{\omega,-})| \leq \varepsilon/4\, .$$
Because $x^{\omega,-}$ is a global minimizer, it is clear that 
$W_{[-M,M]}(X)-W_{[-M,M]}(x^{\omega,-}) \geq 0$. More is true though. To explain this, let us define 
 for $n$ so large that $p_n\geq 2M+2r+2$, the periodic sequences $X^n\in \X_{p_n, q_n}$ by 
\begin{align}\nonumber
X^n_j & :=  \left\{ \begin{array}{ll} x_j & \mbox{if}\ j\in [-M+r,M] \\ x^{n,-}_j &\mbox{if}\ j\in [-\lceil p_n/2\rceil,-M+r-1]\cup [M+1, p_n-\lceil p_n/2\rceil -1] \end{array}\right. 
\end{align}  
Then it holds that $X^n_0=x_0=\xi$ for all $n$ and by construction, $X^n\to X$ pointwise as $n\to \infty$. 
As a result, because $X^n=x^{n,-}$ on a large neighborhood outside $[-M+r,M]$,
\begin{align}
& W_{[-M,M]}(X)-W_{[-M,M]}(x^{\omega,-}) =\lim_{n\to \infty} W_{[-M,M]}(X^n)-W_{[-M,M]}(x^{n,-}) \nonumber \\
\nonumber & = \lim_{n\to \infty} W_{[-\lceil p_n/2\rceil, p_n- \lceil p_n/2\rceil - 1]}(X^n)-W_{[-\lceil p_n/2\rceil, p_n- \lceil p_n/2\rceil - 1]}(x^{n,-}) = \\
 \nonumber & = \lim_{n\to \infty} W_{p_n,q_n}(X^n)-W_{p_n,q_n}(x^{n,-}) \geq \lim_{n\to \infty}P_{p_n/q_n}(\xi)  = \varepsilon >0\, .
\end{align}
We conclude that 
\begin{align}\nonumber
& W_{[-M,M]}(x)-W_{[-M,M]}(Y) \geq  
W_{[-M,M]}(X)-W_{[-M,M]}(x^{\omega,-}) \\ \nonumber & - |W_{[-M,M]}(X)-W_{[-M,M]}(x)| - |W_{[-M,M]}(Y)-W_{[-M,M]}(x^{\omega,-})| \geq  \varepsilon - \varepsilon/4- \varepsilon/4 = \varepsilon/2\, .
\end{align} 
But $Y$ is a variation of $x$ supported in $[-M+r, M]$, so this contradicts our assumption that $x$ is a global minimizer.
This shows that if there is a $x\in \mathcal{M}_{\omega}$ with $x_0=\xi$, then $P_{\omega}(\xi)=0$.
\end{proof}
\noindent Of course, we will not use Theorem \ref{foliationtheorem} until after we proved that the limit $P_{\omega}(\xi)=\lim_{q/p\to\omega} P_{q/p}(\xi)$ really exists.

\section{A fundamental estimate}
The aim of this technical section is to prove Theorem \ref{mainthm}, which provides a fundamental estimate for the periodic Peierls barrier. This estimate will eventually imply that the quasi-periodic Peierls barrier is well-defined and it will lead to a proof of Theorem \ref{th1} of the introduction. 

\subsection{A near-periodicity theorem}\label{nearpersection}
We first prove a preliminary result.
To motivate this result, let us briefly investigate, for some $\omega\in \R$, the linear sequence $x^{\omega}$ defined by $x^{\omega}_j:=x_0+ \omega\cdot j$. It holds for all $j\in \Z$ that 
\begin{align}\label{easycomputation}
\left| x^{\omega}_{j+p} - q - x^{\omega}_j\right|  = |p\omega-q|  \, .
\end{align}
Hence, when $|p\omega -q|$ is small, then  $x^{\omega}$ is {\it nearly} $(p, q)$-periodic, even if $\omega\neq q/p$. Theorem \ref{regularitytheorem} below says that periodic Birkhoff sequences of rotation number $\omega=Q/P$ must possess a similar property as soon as $Q/P$ is close to $q/p$. The precise statement is the following:
\begin{theorem} \label{regularitytheorem}
There is a constant $E\geq 1$, depending only on the range of interaction $r$, for which the following holds. Let $p\neq 0$ and $q$ be relative prime integers, $Q/P$ a rational in lowest terms and $x\in\mathcal{B}_{P,Q}$. Then there is an $i_0\in \Z$ with $-|p| < i_0 \leq 0$ such that
\begin{align}\label{pigeintheorem}
 \sum_{j=i_0}^{i_0+r-1} |x_{j+p} - q - x_{j} | \leq E \left( \frac{1}{|p|} + |p(Q/P) - q| \right)\, .
\end{align}
\end{theorem}
\noindent When $x=x^{Q/P}$ is a linear sequence of rotation number $Q/P$, then estimate (\ref{pigeintheorem}) clearly holds for every $i_0\in \Z$, for $E=r$ and  without the term $1/|p|$. 

A more general result than Theorem \ref{regularitytheorem} was proved in \cite{destruction} - in fact, the theorem is even true for $x\in \mathcal{B}_{\omega}$ with $\omega\notin \Q$. Although Theorem \ref{regularitytheorem} is weaker than the result in \cite{destruction}, it is also easier to prove. We therefore provide this proof here.
 \\ \mbox{}\\
{\it Proof (of Theorem \ref{regularitytheorem})}.
For simplicity and readability, we assume $p>0$ and we prove the theorem in steps. We shall write $\omega:=Q/P$.

\indent {\bf Step 1.} Let $x \in \mathcal{B}_{P,Q}$ be a Birkhoff sequence. We define its {\it hull function} $\psi:\R\to\R$  
  by setting 
$$\psi( k \omega + l ):= x_{k}+l \, .$$
We claim that $\psi$ is well-defined on the set $\{k \omega + l \ | k,l\in \Z\}\subset \R$. This follows from our assumptions:  the equality $k \omega  +l =  K \omega +L$ implies that $(k-K)Q + (l-L)P = 0$ and hence, because $P$ and $Q$ are relative prime, that $k-K= -nP$ and $l-L=nQ$. It thus follows that $x_{k}+l = x_{K-nP} + L+ nQ =x_{K}+L$. So $\psi$ is well-defined.


Most importantly, $\psi$ is nondecreasing: when $k \omega + l > K \omega + L$, then by Proposition \ref{numbertheory} it must hold that $\tau_{-k,l}x>\tau_{-K,L}x$  and in particular, $\psi( k \omega + l) = x_{k}+l =(\tau_{-k,l}x)_0 \geq (\tau_{-K,L}x)_0= x_{K}+L = \psi( K \omega + L)$. 

 It is now clear that $\psi$ can be extended to a nondecreasing map $\psi:\R\to\R$ for which $\psi(\xi+1)=\psi(\xi)+1$ for all $\xi\in\R$.

 {\bf Step 2.} Let us use the map $\psi:\R\to\R$ to define a sequence $y$ by
$$y_i :=  \psi\left(i (q/p) \right)\, .$$
Because $\psi$ is nondecreasing and $\psi(\xi+1)=\psi(\xi)+1$, it follows that  
$$y_{i+k}+l = \psi\left(i (q/p)+ k (q/p) + l\right)  \left\{ \begin{array}{lll} \leq y_i  & \mbox{when} & k (q/p)+l<0\ ,  \\ = y_i  & \mbox{when} & k(q/p)+l=0\ , \\  \geq y_i  & \mbox{when} & k(q/p)+l>0\ .  \end{array} \right.$$
Thus, we observe that $y$ is Birkhoff and, because $-p(q/p)+q=0$, actually $y\in \mathcal{B}_{p,q}$.

{\bf Step 3.} 
Because $p$ and $q$ are relative prime, there exist $s, t\in \Z$ so that 
$$pt-qs=1\, .$$ 
For these integers $s,t$ we claim that 
\begin{align}\label{l1periodic}
||\tau_{s,t}y-y||_{l_1(p)}:= \sum_{i=1}^p |(\tau_{s,t}y)_i - y_i| =1  \ \mbox{for any}\ y\in \B_{p,q}\,  .
\end{align} 
To prove our claim, note that for all $k,l\in \Z$ and $i\in \Z$ it holds that 
$$(\tau_{k,l}^py)_i = (\tau_{pk,pl}y)_i =  y_{i-pk} + pl = y_{i} +(pl-qk)\ \mbox{for all} \ y\in \X_{p,q}\, .$$
In particular, $\tau_{s,t}^py=y+1$ and as a consequence, 
$$||\tau_{s,t}^py-y||_{l_1(p)}  =p\, .$$
If $y\in \B_{p,q}$ is also Birkhoff, then this implies that
%
  $$p = ||\tau_{s,t}^p y - y||_{l_1(p)} = \sum_{j=1}^p||\tau_{s,t}^jy-\tau_{s,t}^{j-1}y||_{l_1(p)} =p ||\tau_{s,t}y - y||_{l_1(p)}\, .$$
 This proves (\ref{l1periodic}). In fact, because $\tau_{k,l}^p y -y\in \Z$ for any $k,l\in \Z$, the fact that $\tau_{s,t}^py-y=1$ just means that $\tau_{s,t}y=\min_{k,l}\{\tau_{k,l}y >y\}$ if $y\in \B_{p,q}$.

 Equation (\ref{l1periodic}) in turn implies that for every Birkhoff $y\in \B_{p,q}$ and every integer $a\in \Z$, there must be a $-p < i_0 \leq 0$ so that,
 \begin{align}\label{firstestimate}
 \sum_{j=i_0}^{i_0+r-1} |(y_{j-as} +at) -(y_{j+as}-at)| = \sum_{j=i_0}^{i_0+r-1} |(\tau_{s,t}^a y)_j-(\tau_{s,t}^{-a}y)_j| \leq \frac{2 |a| r}{p} \, .
 \end{align}
 This follows from the pigeonhole principle. Indeed, if it were true that $\sum_{j=i}^{i+r-1} |(\tau_{s,t}^a y)_j-(\tau_{s,t}^{-a}y)_j|  > \frac{2 |a| r}{p}$ for all $i\in\Z$, then it would follow that 
\begin{align}\nonumber
2 |a| r & < \sum_{i=1}^{p} \sum_{j=i}^{i+r-1} |(\tau_{s,t}^a y)_j-(\tau_{s,t}^{-a}y)_j|  = \sum_{j=0}^{r-1} \sum_{i=1}^{p} |(\tau_{s,t}^a y)_{i+j}-(\tau_{s,t}^{-a}y)_{i+j}| = \\ \nonumber 
 & \sum_{j=0}^{r-1} \sum_{i=1}^{p} |(\tau_{s,t}^a y)_{i}-(\tau_{s,t}^{-a}y)_{i}|  = r ||\tau_{s,t}^ay - (\tau_{s,t}^{-a})y||_{l_1(p)} = 2|a|r\, .
\end{align}
The first equality is a re-summation and the second equality holds because $\tau_{s,t}^{\pm a}y$ are $(p,q)$-periodic. This is a contradiction and we conclude that (\ref{firstestimate}) holds for some $-p< i_0\leq 0$.

{\bf Step 4.} Let us try to find an integer $a>0$ such that
$$j(q/p) - a/p \leq j \omega \leq j(q/p) + a/p\ \mbox{for all}\ -p+1\leq j\leq p+r-1\, .$$ 
It is quite straightforward to check that these inequalities hold for any $a\geq (p+r-1) \left| p\omega - q \right|$ and in particular for
$$a:= \left\lceil (p+r-1) \left|p \omega - q \right|\right\rceil\, . $$ 
Here $\lceil \cdot \rceil$ is the smallest greater integer function. 

Now recall the integers $s, t\in \Z$ for which $pt-qs=1$. 
Using these, we can rewrite $j(q/p) - a/p \!=\! (j+as)(q/p)-at$ and $j(q/p) + a/p \!=\! (j-as)(q/p)+at$. This yields that 
$$(j+as)(q/p)-at \leq j \omega \leq (j-as)(q/p)+at 
\  \mbox{for all}\ -p+1\leq j \leq (p+r-1)\, . $$
Applying the nondecreasing map $\psi$, we then obtain that
\begin{align}\label{secondestimate}
  y_{j+as}-at \leq x_j \leq y_{j-as}+at \ \mbox{for}\ -p+1\leq j\leq p+r-1 \, .
\end{align}
This means that the sequence $x$ is squeezed in between the translates $\tau_{s,t}^{-a}y$ and $\tau_{s,t}^ay$ of the $(p,q)$-periodic sequence $y$ on the segment $[-p+1, p+r-1]$.

 {\bf Step 5.} Given $x$ as in the statement of the theorem, let $y\in \B_{p,q}$ be as constructed in step 2. Then the inequalities (\ref{firstestimate}) and (\ref{secondestimate}) hold for $-p+1\leq j\leq p+r-1$, for some $-p<i_0\leq 0$ and for $a=\left\lceil (p+r-1) \left|p \omega - q \right|\right\rceil$. Because for $i_0\leq j \leq i_0+r-1$ it holds that $-p+1\leq j, p+j \leq p+r-1$, it then follows in particular that
\begin{align}\nonumber
 &y_{j+as}-at \leq x_{j} \leq y_{j-as}+at \ \mbox{for}\ i_0\leq j\leq i_0+ r-1 \ \mbox{and}\\  \nonumber
 &y_{p+j+as}-at \leq x_{p+j} \leq y_{p+j-as}+at \ \mbox{for}\ i_0\leq j\leq i_0+ r-1\, . 
 \end{align}
Subtracting these inequalities, we obtain that
$$(y_{p+j+as}-at) - ( y_{j-as}+at )  \leq x_{p+j} - x_{j} \leq (y_{p+j-as}+at)-(y_{j+as}-at )\, ,$$
or, because $y\in \X_{p,q}$,
$$(y_{j+as}-at) - ( y_{j-as}+at )  \leq x_{p+j} -q - x_{j} \leq (y_{j-as}+at)-(y_{j+as}-at )\, .$$
Summing this over $j=i_0,\ldots, i_0+r-1$ and using (\ref{firstestimate}), we then obtain
$$\sum_{j=i_0}^{i_0+r-1} | x_{p+j} -q - x_{j}  | \leq \frac{2ar}{p} = \frac{2r\left\lceil (p+r-1) \left|p \omega - q \right|\right\rceil}{p} \leq E \left( \frac{1}{p} + |p\omega-q|\right)\, . $$
Here, $E = 2r^2$ and we have used that $\lceil x\rceil \leq 1+x$ and that $\frac{p+r-1}{p}\leq r$.

This finishes the proof of Theorem \ref{regularitytheorem}.
\hfill $\square$

\subsection{A comparison of periodic Peierls barriers}\label{proof of tl2}
The following theorem is the main technical result of this paper. It provides a comparison between the Peierls barrier functions of different rational rotation numbers. 
\begin{theorem}\label{mainthm}
For each $L>0$ there is a $C>0$ such that for all rationals $q/p$ and $Q/P$ in lowest terms with $\left|\frac{q}{p}\right|, \left|\frac{Q}{P}\right| \leq L$, 
\begin{align}\label{basicperiodicestimate}
\sup_{\xi\in \R} \left| P_{Q/P}(\xi) - P_{q/p}(\xi) \right| \leq C \left( \frac{1}{|p|} + |p (Q/P) - q| \right) \, .
\end{align}
\end{theorem}

\begin{proof} For readability, we shall assume that $p,P>0$. Let us start the proof with a few general remarks. First of all, we note that it is enough to prove the theorem in case 
\begin{align}\label{assumption}
|p(Q/P)-q| < 1\, .
\end{align}
To understand this, let $x^-\in \M_{p,q}\subset \X_{L+2}$ be any periodic minimizer with $|x_0-\xi|\leq 1$ and define $x^{\xi} \in \X_{p,q}$ by $ x^{\xi}_0:=\xi$ and $x^{\xi}_j=x^-_j$ for $j=1, \ldots, p-1$. Then it follows that $x^{\xi}\in \X_{L+3}$ and thus from Proposition \ref{lipschitz} that $|W_{p,q}(x^{\xi})-W_{p,q}(x^-)| \leq Dr$. As a result, 
$$P_{p,q}(\xi) \leq W_{p,q}(x^{\xi}) - W_{p,q}(x^-) \leq Dr\, .$$ 
Similarly, $P_{Q/P}(\xi) \leq Dr$ and in particular, $|P_{Q/P}(\xi)-P_{q/p}(\xi)| \leq 2Dr$. But this means that (\ref{basicperiodicestimate}) will certainly hold if $|p(Q/P)-q|\geq 1$ for any $C\geq 2Dr$. So we will assume that (\ref{assumption}) is satisfied throughout this proof.

In turn, assumption (\ref{assumption}) implies that we may assume  that $p\neq P$. Indeed, if $p=P$ then either $q=Q$ and there is nothing to prove or $q\neq Q$ and $|p(Q/P)-q|\geq 1$. In fact, it suffices to prove the theorem only for $P>p$. Namely, if $P>p$, then  
$$\frac{1}{p} + \left|p\frac{Q}{P}-q\right|  = \frac{1}{p}+\frac{1}{P} +\frac{\left| pQ-qP \right| - 1}{P} \leq \frac{1}{P}+\frac{1}{p} +\frac{\left| pQ-qP \right| - 1}{p}  =
\frac{1}{P}+\left|P\frac{q}{p}-Q\right|$$
so if $P>p$, then (\ref{basicperiodicestimate}) implies the estimate (\ref{basicperiodicestimate}) with $q/p$ and $Q/P$ interchanged. Hence our second assumption, that
$$P>p>0\, .$$
Now we start the proof by fixing $\xi \in \R$ and letting $x^-\in \mathcal{M}_{p,q}$, $y^-\in\mathcal{M}_{P,Q}$ and $z^-\in \M_{P-p, Q-q}$ be periodic minimizers and
$x(\xi)\in \mathcal{M}_{p,q}(\xi)$ and $y(\xi)\in \mathcal{M}_{P,Q}(\xi)$ periodic $\xi$-minimizers. For later use, we remark that $x^-, y^-, x(\xi), y(\xi)\in \X_{L+2}$. 
Moreover, it follows from (\ref{assumption}) that 
\begin{align}\label{consequence} 
\left|\frac{Q}{P} - \frac{Q-q}{P-p}\right| \leq \frac{1}{P-p} \leq 1\, 
\end{align}
and as a consequence $\frac{Q-q}{P-p}\leq L+1$. This implies that $z^-\in\X_{L+3}$.
\\ \mbox{} \\
\noindent {\bf An estimate from below:} We will first estimate $P_{Q/P}(\xi) - P_{q/p}(\xi)$ from below. To do so, recall from Theorem \ref{regularitytheorem} that there is a $-p< i_0 \leq 0$ with the property that 
\begin{align}\label{periodicestimate}
\sum_{j=i_0}^{i_0+r-1} |y(\xi)_{j+p} - q - y(\xi)_{j}| \leq E \left(\frac{1}{p} + |p(Q/P) - q|\right)\, .
\end{align}
Using this $i_0$, we define
$$B:=[i_0, i_0+P-1]=B_{1}\cup B_{2} \ \mbox{with}\ B_{1}:=[i_0,i_0+p-1] \ \mbox{and}\ B_{2}:=[i_0+p, i_0+P-1]$$
and according to this decomposition, we write
\begin{align}&
P_{Q/P}(\xi)-P_{q/p}(\xi) = W_{B}(y(\xi)) - W_{B}(y^-) - W_{B_1}(x(\xi)) + W_{B_1}(x^-) = \nonumber \\ \nonumber
& \underbrace{W_{B_1}(y(\xi)) - W_{B_1}(x(\xi))}_{\bf (1)} +\underbrace{W_{B_2}(y(\xi)) -W_{B_2}(z^-)}_{\bf (2)} + \underbrace{W_{B_1}(x^-)  + W_{B_2}(z^-) - W_{B}(y^-)}_{\bf (3)}  \, .
\end{align}
We will estimate terms $({\bf 1}), ({\bf 2})$ and $({\bf 3})$ separately from below. 

To estimate term ({\bf 1}), we use that $y(\xi)$ is nearly $(p,q)$-periodic. This means the following: let us denote by $\tilde y(\xi)\in \X_{p,q}$ the $(p,q)$-periodic approximation of $y(\xi)$ defined by
$$\tilde y(\xi)_j: = y(\xi)_j \ \mbox{for} \ j\in B_1=[i_0, i_0+p-1]\, . $$
We claim that $\tilde y(\xi)_j$ is very close to $y(\xi)_j$ for all $j\in [i_0, i_0+p-1+r]$. To prove this, let us remark first of all that $\tilde y(\xi)_j= y(\xi)_{j-p}+q$ for all $j\in [i_0+p, i_0+2p-1]$. It therefore follows from (\ref{periodicestimate}) that 
\begin{align}\nonumber
&\sum_{j= i_0}^{i_0+p-1+\min\{p,r\}} \!\!\! | y(\xi)_j - \tilde y(\xi)_j | = \sum_{j= i_0+p}^{i_0+p-1+\min\{p,r\}} \!\!\! | y(\xi)_j - \tilde y(\xi)_j | =  \nonumber \\ \nonumber & \sum_{j= i_0+p}^{i_0+p-1+\min\{p,r\}} \!\!\! | y(\xi)_j -(y(\xi)_{j-p}+q) | =
 \sum_{j=i_0}^{i_0-1+\min\{p,r\}} \!\!\! | y(\xi)_{j+p} - q - y(\xi)_j | \\ \nonumber &\leq E \left(\frac{1}{p} + |p(Q/P) - q|\right)\, . 
\end{align}
When $p\geq r$, this proves our claim that $\tilde y(\xi)_j$ is very close to $y(\xi)_j$ for all $j\in [i_0, i_0+p-1+r]$. Otherwise, in the (exceptional) case that $p<r$, the estimate is similar but more delicate. In fact, one can estimate for any integer $m\geq 1$,
\begin{align}\nonumber
&\sum_{j= i_0+mp}^{i_0+p-1+\min\{mp,r\}} \!\!\! | y(\xi)_j - \tilde y(\xi)_j | = \sum_{j=i_0+mp}^{i_0+p-1+\min\{mp,r\}} \!\!\! | y(\xi)_{j} - mq - y(\xi)_{j-mp} | \nonumber \leq \\
\nonumber &  \sum_{j=i_0+mp}^{i_0+p-1+\min\{mp,r\}} \!\!\! | y(\xi)_{j} -q -y(\xi)_{j-p}|  + \ldots +  | y(\xi)_{j-(m-1)p}  -q - y(\xi)_{j-mp} | \leq \nonumber \\ \nonumber
&  \sum_{j=i_0+p}^{i_0+p-1+\min\{mp,r\}} \!\!\! | y(\xi)_{j} -q -y(\xi)_{j-p}| = \sum_{j=i_0}^{i_0-1+\min\{mp,r\}} \!\!\! | y(\xi)_{j+p} -q -y(\xi)_j|  \\ \nonumber & \leq
E \left(\frac{1}{p} + |p(Q/P) - q|\right)\, . 
\end{align}
Summing these inequalities for $m=1, \ldots, r$, they yield that 
\begin{align}\label{anotherestimate}
&
\sum_{j= i_0}^{i_0+p-1+r} \!\!\! | y(\xi)_j - \tilde y(\xi)_j |  \leq rE \left(\frac{1}{p} + |p(Q/P) - q|\right)\, . 
\end{align}

\noindent We also remark that $|\tilde y(\xi)_{i_0-1} - \tilde y(\xi)_{i_0}| = | y(\xi)_{i_0+p-1}-q - y(\xi)_{i_0}| \leq  | y(\xi)_{i_0+p-1} -  y(\xi)_{i_0+p}| + | y(\xi)_{i_0+p}-q -  y(\xi)_{i_0}| \leq (L+2)+E(1/p+|p(Q/P)-q|) \leq L+2+2E$ so $\tilde y(\xi)\in \X_K$ for $K:=L+2+2E$. Thus, Proposition \ref{lipschitz} and (\ref{anotherestimate}) together give that
$$|W_{ B_{1}}(\tilde y(\xi))-W_{ B_{1}}(y(\xi))|\leq rDE \left( \frac{1}{p} + |p(Q/P) -q|\right) \, .$$ 
Because, by construction, $\tilde y(\xi)_0=\xi$, and because $x(\xi)$ minimizes $W_{B_1}=W_{p,q}$ over $\X_{p,q}$ subject to $\xi$, it holds that $W_{B_1}(\tilde y(\xi)) - W_{B_1}(x(\xi))\geq 0$, so this proves that 
$$({\bf 1}) \geq - rDE\left( \frac{1}{p} + |p(Q/P)-q| \right)\, .$$
In order to estimate term $({\bf 2})$, let us analyse how $y(\xi)$ behaves on $B_{2}$. We claim that its restriction to $[i_0+p, i_0+P-1+r]$ is almost $(P-p, Q-q)$-periodic. Indeed, this follows from (\ref{periodicestimate}) and from the fact that $y(\xi)$ is $(P,Q)$-periodic:
\begin{equation}\begin{aligned} 
\sum_{j= i_0+p}^{i_0+p+r-1} |y(\xi)_{j+(P-p)}-(Q-q) - y(\xi)_j| = \sum_{j= i_0}^{i_0+r-1} |y(\xi)_{j+P}-(Q-q)-y(\xi)_{j+p}| = \nonumber \\ \nonumber
 \sum_{j= i_0}^{i_0+r-1} |y(\xi)_{j}+q-y(\xi)_{j+p}| =  \sum_{j= i_0}^{i_0+r-1} |y(\xi)_{j+p}-q-y(\xi)_{j}|\leq E\left(\frac{1}{p} + |p(Q/P) - q|\right) \, .\end{aligned}
\end{equation} 
With this in mind, the analysis of term $({\bf 2})$ proceeds as the analysis of term ${\bf (1)}$: this time we let $\hat y(\xi)\in \X_{P-p, Q-q}$ be a $(P-p,Q-q)$-periodic approximation of $y(\xi)$, that is 
$$\hat y(\xi)_i: = y(\xi)_i\ \mbox{for} \ i\in B_2= [i_0+p, i_0+P-1]\, .$$ 
As above, it then holds that 
\begin{align}\nonumber
&
\sum_{j= i_0+p}^{i_0+P-1+r} \!\!\! | y(\xi)_j - \hat y(\xi)_j |  \leq rE \left(\frac{1}{p} + |p(Q/P) - q|\right)\, . 
\end{align}
and that 
$\hat y(\xi)\in \X_{L+2+2E}$.  Therefore,
$$\left| W_{B_{2}}(\hat y(\xi)) - W_{B_{2}}(y(\xi)) \right|  \leq rDE \left( \frac{1}{p} + |p(Q/P) -q|\right)\, .$$
Because $z^-\in \X_{P-p,Q-q}$ is a $(P-p,Q-q)$-periodic minimizer, this yields that
$$({\bf 2}) \geq - rDE \left( \frac{1}{p} + |p(Q/P)-q| \right)\, .$$
Finally, we estimate term $({\bf 3})$ by defining, for any $n\in \N$, a sequence $Y\in \X_{nP,nQ}$ by 
$$Y_i:=\left\{ \begin{array}{ll} x^-_i & \mbox{for}\ i\in [i_0, i_0+np-1] \\ z_i^- & \mbox{for}\ i\in [i_0+np, i_0+nP-1]\end{array} \right.$$
 By translation-invariance it may of course be assumed that $|x^-_{i_0+np} - z^-_{i_0+np}|\leq 1$. Recalling that $x^-\in \X_{L+2}$ and $z^-\in \X_{L+3}$, it then follows that $|Y_{i_0+np-1}-Y_{i_0+np}| = |x^-_{i_0+np-1}-z^-_{np}| \leq L+3$ but also that $|Y_{i_0-1}-Y_{i_0}| = |z^-_{i_0+nP-1} - nQ -x^-_{i_0}| = |z^-_{i_0+nP-1-n(P-p)}+n(Q-q)-nQ - x_{i_0+np}^- +nq| = |z^-_{i_0+np-1} - x^-_{i_0+np}| \leq L+4$. The conclusion is that $x^-, z^-, Y\in \X_K$ for $K := L+4$. As a consequence, 
\begin{align}
&\sum_{j\in [i_0,i_0+np-1+r]}\!\!\! |x^-_j-Y_j| = \!\!\! \sum_{j\in [i_0+np,i_0+np-1+r]}\!\!\! |x^-_j-Y_j|  \leq 2K + \ldots + 2rK \leq r(r+1)K \  \mbox{and} \nonumber \\ \nonumber &\sum_{j\in [i_0+np,i_0+nP-1+r]}\!\!\!|z^-_j-Y_j| =\!\!\!\sum_{j\in [i_0+nP,i_0+nP-1+r]}\!\!\! |z^-_j-Y_j|   \leq 2K + \ldots + 2rK \leq r(r+1)K \, .
\end{align}
 It therefore follows from Proposition \ref{lipschitz} that
\begin{align}
& |W_{[i_0, i_0+np-1]}(x^-)-W_{[i_0, i_0+np-1]}(Y)| \leq r(r+1)DK\ \ \mbox{and} \nonumber \\ \nonumber & |W_{[i_0+np, i_0+nP-1]}(z^-)-W_{[i_0+np, i_0+nP-1]}(Y)| \leq r(r+1)DK\, 
\end{align}
and as a result, because $y^-\in \M_{P,Q}=\M_{nP,nQ}$ is a periodic minimizer,
\begin{align}
& n \left( W_{[p,q]}(x^-) + W_{P-p,Q-q}(z^-) - W_{P,Q}(y^-) \right) = \nonumber \\ \nonumber & W_{[i_0,i_0+np-1]}(x^-) + W_{[i_0+np, i_0+nP-1]}(z^-) - W_{nP,nQ}(y^-)  
\geq \nonumber \\ \nonumber & W_{[i_0,i_0+np-1]}(Y)+W_{[i_0+np, i_0+nP-1]}(Y) - W_{nP,nQ}(y^-) - 2r(r+1)DK  = \nonumber \\ \nonumber & W_{nP,nQ}(Y) - W_{nP,nQ}(y^-) - 2r(r+1)DK  \geq -2r(r+1)DK  
 \, .
\end{align}
But this is true for all $n\in \N$ and we conclude that
$$({\bf 3}) \geq 0 \, .$$
This completes the proof that $$P_{Q/P}(\xi) - P_{q/p}(\xi) \geq -2r DE \left( \frac{1}{p} + |p(Q/P) -q|\right)\, .$$
\\ \mbox{} \\
\noindent {\bf An estimate from above:} To estimate $P_{Q/P}(\xi)-P_{q/p}(\xi)$ from above, we follow a similar procedure, but this time we choose $-p < i_0 \leq 0$ in such a way that 
\begin{align}
\sum_{j=i_0}^{i_0+r-1} |y^-_{j+p} - q - y^-_{j}| \leq E \left(\frac{1}{p} + |p(Q/P) - q|\right)
\end{align}
and we write
\begin{align}&
-\left(P_{Q/P}(\xi) -  P_{q/p}(\xi) \right)=  W_{B}(y^-) - W_{B}(y(\xi)) + W_{B_1}(x(\xi)) -W_{B_1}(x^-) = \nonumber \\ \nonumber
& \underbrace{W_{B_1}(y^-) - W_{B_1}(x^-)}_{\bf (1')} +\underbrace{W_{B_2}(y^-) -W_{B_2}(z^-)}_{\bf (2')} + \underbrace{W_{B_1}(x(\xi))  + W_{B_2}(z^-) - W_{B}(y(\xi))}_{\bf (3')}  \, .
\end{align}
As above, one finds that
$$({\bf 1'}), ({\bf 2'}) \geq -rDE\left( \frac{1}{p} + |p(Q/P)-q|\right)$$ 
respectively because $x^-\in \M_{p,q}$ is a periodic minimizer and $y^-$ is almost $(p,q)$-periodic and because $z^-\in \M_{P-p, Q-q}$ is a periodic minimizer and $y^-$ is almost $(P-p, Q-q)$-periodic. It also follows that $|({\bf 3'})|\geq 0$ because $y(\xi)\in \M_{P,Q}(\xi)$ is a periodic $\xi$-minimizer. 

This concludes the proof of (\ref{basicperiodicestimate}) for $C:=2rDE$, where $D=(r+1) ||S||_{C^1(\X_K)}$, $K=L+2+2E$ and $E=2r^2$.
\end{proof}

\section{Continuity of the Peierls barrier}\label{modulussection}
In this section, we will use the estimate of Theorem \ref{mainthm} to show, for $\omega\in \R\backslash \Q$, that the limit $P_{\omega}(\xi)=\lim_{q/p\to\omega} P_{q/p}(\xi)$
is well-defined. We will also prove Theorem \ref{th1} of the introduction. 

We simply follow the argument given in \cite{Matherpeierls}, which goes as follows. First of all, we will say that a rational number $q/p$ is a {\it best rational approximation} of an irrational number $\omega$ if 
$$|p\omega-q|\leq |p'\omega-q'| \ \mbox{for any} \ 0\leq |p'| \leq |p|\, .$$ 
For irrational $\omega$, there clearly exist infinitely many such best rational approximations. Moreover, the following result is well-known:
\begin{proposition}
Let $q/p$ and $\mathfrak{q}/\mathfrak{p}$ be {\rm successive} best rational approximations of $\omega\in \R\backslash\Q$. This means that $|\mathfrak{p}|>|p|$ and $|\mathfrak{p}\omega-\mathfrak{q}| < |p\omega - q| \leq |p'\omega - q'|$ for all $0\leq |p'|<|\mathfrak{p}|$. Then  
$$\left|p\omega- q\right|<\frac{1}{|\mathfrak{p}|}<\frac{1}{|p|}\, .$$ 
\end{proposition}
\begin{proof}
For any $\varepsilon>0$, the parallelogram   
$$\{(x,y)\in \R^2\, |\, -|\mathfrak{p}| < x<|\mathfrak{p}| \ \mbox{and} \ |\omega x-y| < (1+\varepsilon)/|\mathfrak{p}|\}$$
has volume equal to $4(1+\varepsilon)$ and hence by Minkowski's theorem must contain at least one nontrivial integer point. But for $-|\mathfrak{p}|<p'<|\mathfrak{p}|$, the quantity $|p'\omega-q'|$ is minimized at $p'=p$ and $q'=q$. Hence the parallelogram must certainly contain the point $(p,q)$. This means that $|p\omega-q|<(1+\varepsilon)/|\mathfrak{p}|$. The latter is true for all $\varepsilon>0$, so $|p\omega-q|\leq 1/|\mathfrak{p}|$ and because $\omega$ is irrational, actually $|p\omega-q|<1/|\mathfrak{p}|<1/|p|$.
\end{proof}
 
\noindent We will use these best rational approximations to prove the following result:
\begin{corollary}\label{limitcorr}
For $\omega\in\R\backslash \Q$ the limit $P_{\omega}(\xi)=\lim_{Q/P\to\omega} P_{Q/P}(\xi)$ exists and satisfies
$$|P_{\omega}(\xi)-P_{q/p}(\xi)|\leq C \left( \frac{1}{|p|} + |p\omega - q|\right)\ \mbox{for all}\ q/p\ \mbox{with} \ |\omega-q/p| \leq 1\, .$$
\end{corollary}
\begin{proof}
Let $\omega$ be irrational and let $q/p$ be a best rational approximation of $\omega$. This implies that $|\omega - q/p| < 1/|p|^2$. Therefore, when $|Q/P-\omega| < 1/|p|^2$, then $|Q/P-q/p| \leq |Q/P - \omega| + |\omega - q/p| < 2/|p|^2$ and hence it follows from Theorem \ref{mainthm} that 
$$|P_{Q/P}(\xi)-P_{q/p}(\xi)| \leq C \left(\frac{1}{|p|}+|p(Q/P)-q|\right) \leq 3C/|p|\, .$$ 
But of course $|p|$ can be chosen arbitrarily large, so this shows that the limit $P_{\omega}(\xi)=\lim_{Q/P\to\omega} P_{Q/P}(\xi)$ exists. 

Obviously, for $\omega\notin\Q$ and $q/p$ arbitrary,
\begin{align}
|P_{\omega}(\xi)-&P_{q/p}(\xi)|\leq \lim_{Q/P\to\omega} |P_{\omega}(\xi)-P_{Q/P}(\xi)| + |P_{Q/P}(\xi)-P_{p/q}(\xi)| \leq  \nonumber \\ \nonumber & \lim_{Q/P\to\omega} C\left( \frac{1}{|p|}+|p(Q/P) -q| \right) =  C \left( \frac{1}{|p|} + |p\omega - q|\right)\, .
\end{align}
\end{proof}
\noindent It follows from Corollary \ref{limitcorr} that $P_{\omega}(\xi)$ depends continuously on $\omega$ at irrational $\omega$. 
Indeed, when $\omega$ is irrational, $q/p$ is one of its best rational approximations and $\Omega$ is another rotation number with $|\Omega-\omega|\leq 1/|p|^2$, then $|p\Omega - q| \leq |p\omega -q| + |p(\Omega-\omega)| \leq 2/|p|$ and therefore
\begin{align}
|P_{\omega}(\xi)& -P_{\Omega}(\xi)|\leq |P_{\omega}(\xi)-P_{q/p}(\xi)| + |P_{\Omega}(\xi)-P_{q/p}(\xi)| \leq \nonumber \\ \nonumber 
 C& (1/|p|+|p\omega-q|) + C(1/|p|+|p\Omega-q|) \leq 5C/|p|\, .
\end{align}
This proves:
\begin{corollary}
The map $\omega\mapsto P_{\omega}(\xi)$ is continuous at irrational $\omega$ (uniformly in $\xi$).
\end{corollary} \noindent
More quantitative continuity results may be obtained under certain conditions on the irrationality of $\omega$. For example, we recall that $\omega$ is called {\it Diophantine} when there are constants $\gamma>0$ and $\tau\geq 1$ such that 
$$|p\omega -q| \geq \frac{\gamma}{|p|^{\tau}} \ \mbox{for all integers}\ p\neq 0\ \mbox{and}\ q \, .$$
\begin{corollary}
When $\omega$ is Diophantine, then uniformly in $\xi$ the map $\omega\mapsto P_{\omega}(\xi)$ is locally H\"older continuous at $\omega$ with H\"older exponent $1/2\tau$:
$$\left| P_{\omega}(\xi) - P_{\Omega}(\xi) \right|  < (5C/ \gamma^{1/\tau}) \, |\Omega-\omega|^{1/2\tau}\ \mbox{whenever}\ |\Omega-\omega| \leq 1\, .$$
\end{corollary}
\begin{proof}
Let $\omega$ be Diophantine and $\Omega$ arbitrary with $|\Omega-\omega|\leq 1$ and choose two successive best rational approximations $q/p$ and $\mathfrak{q}/\mathfrak{p}$ of $\omega$ in such a way that $1/|\mathfrak{p}|^2 < |\Omega-\omega|\leq 1/|p|^2$. 

Then it follows that 
$$|p\omega-q|< \frac{1}{|\mathfrak{p}|} \ \mbox{and}\ |p\omega-q|\geq \frac{\gamma}{|p|^{\tau}}\ \mbox{and hence}\  |\mathfrak{p}| < |p|^{\tau}/\gamma\, .$$
As a result,
$$\left| P_{\omega}(\xi) - P_{\Omega}(\xi) \right|  \leq 5C/|p| < (5C/\gamma^{1/\tau})/|\mathfrak{p}|^{1/\tau} < (5C/ \gamma^{1/\tau}) \, |\Omega-\omega|^{1/2\tau}\, .$$
\end{proof}
\noindent The results in this section together prove Theorem \ref{th1}.

\section{Continuous dependence on parameters}\label{section robustness}
We now prove that the Peierls barrier $P_{\omega}(\xi)$ depends continuously on the local action:
\begin{proposition}\label{continuous}
Fix $\omega \in \R$ and let $S$ be a local potential satisfying conditions {\bf A}-{\bf C} of Section \ref{setting}. For all $\varepsilon>0$ there are a $\delta>0$ and a $K>0$ so that whenever $S^{\delta}$ satisfies conditions {\bf A}-{\bf C} and the estimates $\|S^{\delta} - S\|_{C^0(\X_{K})} < \delta$ and $\|S^{\delta} - S\|_{C^1(\X_K)} \leq 1$, then 
$$\sup_{\xi\in \R} | P_{\omega}^{\delta}(\xi) - P_{\omega}(\xi)| < \varepsilon\, .$$ 
\end{proposition}

\begin{proof} Let $\varepsilon>0$ be given. We first let $\omega=Q/P$ be rational, choose $K$ so that $\mathcal{B}_{P,Q}\subset \X_K$ and assume that $||S^{\delta}-S||_{C^0(\X_K)} < \delta:= \frac{\varepsilon}{4|P|}$. Then it holds for all $x\in \B_{P,Q}$ that

\begin{equation}\label{sp1}|W^{\delta}_{P,Q}(x)-W_{P,Q}(x)|< |P|\delta = \frac{\varepsilon}{4}\, .\end{equation} 
Denoting by $x^-\in \mathcal{M}_{P,Q}$ a periodic minimizer of $W_{P,Q}$, by $x^{\delta,-}\in \mathcal{M}_{P,Q}^{\delta}$ a periodic minimizer of $W^{\delta}_{P,Q}$, by $x(\xi)\in \mathcal{M}_{P,Q}(\xi)$ a periodic $\xi$-minimizer of $W_{P,Q}$ and by $x^{\delta}(\xi)\in \mathcal{M}^{\delta}_{P,Q}(\xi)$ a periodic $\xi$-minimizer of $W^{\delta}_{P,Q}$, it then follows from minimality that
\begin{align*}& 
W_{P,Q}^{\delta}(x^{\delta,-})\leq W_{P,Q}^{\delta}(x^{-}) \leq W_{P,Q}(x^{-})+ \frac{\varepsilon}{4}\, ,\\& 
W_{P,Q}(x^{-})\leq W_{P,Q}(x^{\delta,-}) \leq W_{P,Q}^{\delta}(x^{\delta,-})+ \frac{\varepsilon}{4}\, , \\& 
W_{P,Q}^{\delta}(x^{\delta}(\xi))\leq W_{P,Q}^{\delta}(x(\xi)) \leq W_{P,Q}(x(\xi))+ \frac{\varepsilon}{4}\, ,\\& 
W_{P,Q}(x(\xi))\leq W_{P,Q}(x^{\delta}(\xi)) \leq W_{P,Q}^{\delta}(x^{\delta}(\xi))+ \frac{\varepsilon}{4}\, .
\end{align*}

\noindent From this it clearly follows that
\begin{equation}\label{sp2}|P_{Q/P}^{\delta}(\xi)-P_{Q/P}(\xi)| \leq |W_{P,Q}^{\delta}(x^{\delta}(\xi)) - W_{P,Q}(x(\xi))|+|W_{P,Q}^{\delta}(x^{\delta,-}) - W_{P,Q}(x^{-}) |   \leq \frac{\varepsilon}{2}\, .
\end{equation}
This proves the proposition when $\omega=Q/P$ is rational. 

In case $\omega$ is irrational, we choose $K$ large enough that all the sequences constructed in the proof of Theorem \ref{mainthm} are in $\X_K$ (in fact, one may check that $K=|\omega|+3+2r^2$ suffices). Then the proofs of Theorem \ref{mainthm} and Corollary \ref{limitcorr} show that there is a constant $C$ depending only on $||S||_{C^1(\X_K)}$ for which 
$$|P_{\omega}(\xi)-P_{q/p}(\xi)|\leq C \left( \frac{1}{|p|} + |p\omega - q|\right)\ \mbox{for all}\ q/p\ \mbox{with} \ |\omega-q/p| \leq 1\, .$$
In particular we may assume that this estimate holds both for the Peierls barrier of $S$ and for the Peierls barrier of $S^{\delta}$, because $||S^{\delta}||_{C^1(\X_K)} \leq ||S||_{C^1(\X_K)}+1$.  It then holds for any best rational approximation $q/p$ of $\omega$ and any $Q/P$ with $|\omega-Q/P|<1/|p|^2$ that
\begin{align}\label{blabla}
|P_{\omega}(\xi)-P_{Q/P}(\xi)| \leq 5C/|p|\ \mbox{and}\ |P^{\delta}_{\omega}(\xi)-P^{\delta}_{Q/P}(\xi)| \leq 5C/|p|\, .
\end{align}
\noindent 
It now follows from (\ref{sp2}) and (\ref{blabla}) that
$$|P^{\delta}_{\omega}(\xi)-P_{\omega}(\xi)| \leq |P^{\delta}_{\omega}(\xi) - P^{\delta}_{Q/P}(\xi)| +|P^{\delta}_{Q/P}(\xi) - P_{Q/P}(\xi)| +  |P_{Q/P}(\xi) - P_{\omega}(\xi)| \leq \frac{10C}{|p|}+\frac{\varepsilon}{2}\, .$$
Thus, if we choose $q/p$ so that $10C/|p|< \varepsilon/2$ (and $Q/P$ so that $|\omega-Q/P|<1/|p|^2$ and $\delta=\varepsilon/4|P|$), then it follows that $|P^{\delta}_{\omega}(\xi)-P_{\omega}(\xi)| < \varepsilon$. 
\end{proof}

\noindent The main corollary of Proposition \ref{continuous} is Theorem \ref{robustness} below, which was formulated a bit more weakly as Theorem \ref{thm1intro} in the introduction. It says that the collection of local potentials that do not admit a foliation of a specific rotation number, is open in the $C^1$-topology. Thus, one could say that laminations are ``robust''.
\begin{theorem}\label{robustness}
Let $\omega \in \R$ and let $S$ be a local potential satisfying conditions {\bf A}-{\bf C} of Section \ref{setting}. Assume that $\M_\omega$ is not a foliation. Then there exist a $\delta>0$ and a $K>0$ such that for all local potentials $S^{\delta}$ that satisfy conditions {\bf A}-{\bf C} and the estimates $\|S^{\delta} - S\|_{C^0(\X_{K})} < \delta$ and $\|S^{\delta} - S\|_{C^1(\X_{K})} \leq 1$, also $\M_\omega^\delta$ is not a foliation.
\end{theorem}
\begin{proof}
Because $\M_{\omega}$ is not a foliation, by Theorem \ref{foliationtheorem} there is a $\xi$ so that $P_{\omega}(\xi)>0$. By Proposition \ref{continuous} there then exist a $\delta>0$ and a $K>0$ so that $P_{\omega}^{\delta}(\xi)>0$ for all local potentials $S^{\delta}$ with $\| S^{\delta}-S\|_{C^0(\X_K)} < \delta$ and $\| S^{\delta}-S\|_{C^1(\X_K)} \leq 1$. By Theorem \ref{foliationtheorem}, such $S^{\delta}$ do not admit a foliation of rotation number $\omega$ either. 
\end{proof}

\begin{small}
\bibliographystyle{amsplain}
\bibliography{lattice}
\end{small}

\end{document}